\documentclass[12pt,a4paper]{amsart}
\usepackage{amsmath}
\usepackage{amssymb}
\usepackage{amsthm}
\usepackage[english]{babel}
\usepackage{verbatim}
\usepackage[shortlabels]{enumitem}
\usepackage{tikz-cd}
\usepackage{etoolbox} 

\newtheoremstyle{mio}%
	{}{} 
	{\itshape}{} 
	{\bfseries}{.}{ } 
	{#1 #2\thmnote{~\mdseries(#3)}} 
\theoremstyle{mio}
\newtheorem{teor}{Theorem}[section]
\newtheorem{cor}[teor]{Corollary}
\newtheorem{prop}[teor]{Proposition}
\newtheorem{lemma}[teor]{Lemma}

\theoremstyle{definition}

\newtheorem{oss}[teor]{Remark}

\setenumerate{noitemsep}

\newcommand{\Br}{\mathrm{Br}}

\newcommand{\ins}[1]{\mathbb{#1}}
\newcommand{\insN}{\ins{N}}
\newcommand{\insZ}{\ins{Z}}

\newcommand{\inN}{\in\insN}
\newcommand{\inZ}{\in\insZ}

\newcommand{\N}{\mathbb{N}}

\newcommand{\inslim}{\mathcal{L}}

\newcommand{\Int}{\mathrm{Int}}

\newcommand{\bdelta}{\boldsymbol{\delta}}

\title{The polynomial closure is not topological}
\author{Giulio Peruginelli}
\date{\today}
\address{Dipartimento di Matematica, Universit\`a di Padova, Padova, Italy}
\email{gperugin@math.unipd.it}
\author{Dario Spirito}
\address{Dipartimento di Matematica, Universit\`a di Padova, Padova, Italy}
\email{spirito@math.unipd.it}
\curraddr{Dipartimento di Scienze Matematiche, Informatiche e Fisiche, Universit\`a di Udine, Udine, Italy}
\email{dario.spirito@uniud.it}

\begin{document}
\begin{abstract}
We characterize the polynomial closure of a pseudo-convergent sequence in a valuation domain $V$ of arbitrary rank, and then we use this result to show that the polynomial closure is never topological when $V$ has rank at least $2$. 
\end{abstract}

\maketitle

\section{Introduction}
Let $D$ be an integral domain with quotient field $K$ and let $S\subseteq K$ be a subset. The ring of \emph{integer-valued polynomial} over $S$ is
\begin{equation*}
\Int(S,D):=\{f\in K[X]\mid f(S)\subseteq D\}.
\end{equation*}
The \emph{polynomial closure} of $S$, denoted by $\overline{S}$, is the largest subset of $K$ for which the equality $\Int(S,D)=\Int(\overline{S},D)$ holds, and a subset $S$ is \emph{polynomially closed} if $S=\overline S$.

Chabert studied in \cite{chabert-polynomialclosure} conditions under which the polynomial closure is topological, i.e., when there is a topology on $K$ whose closure operator is the polynomial closure; he showed that for this to happen $D$ must be a local domain, and $D=V$ a valuation domain of rank $1$ is a sufficient condition. The purpose of this paper is to complement the latter result by showing that, when $V$ is a valuation domain of rank bigger than $1$, the polynomial closure is \emph{never} topological.

We prove this result by means of the subsets that Chabert used for his own. Indeed, Chabert described the polynomial closure of a generic subset $S$ of $V$ by using \emph{pseudo-convergent sequences}, originally introduced by Ostrowski to study extensions of valued fields \cite{ostrowski-pcv-1} and later used by Kaplansky in the study of maximal fields \cite{kaplansky-maxfield} (see below for the definitions), as well as new related classes of \emph{pseudo-divergent} and \emph{pseudo-stationary} sequences which he introduced; more precisely, he showed that the polynomial closure of $S$ can be described by adding all the \emph{pseudo-limits} of the sequences of these kinds contained in $S$ \cite[Theorem 5.2]{chabert-polynomialclosure}; these three types of sequences can also be used to generalize the results of Ostrowski \cite{fundstatz}. In this paper, we completely describe the polynomial closure of a pseudo-convergent sequence for valuation domains of arbitrary rank; this will allow to show that, for some explicitly constructed pseudo-convergent sequence $E:=\{s_n\}_{n\inN}$, we have $\overline{E}\neq\overline{\{s_1\}}\cup\overline{E\setminus\{s_1\}}$, and thus that the polynomial closure is not topological.

\medskip

Throughout the article, we assume that $V$ is a valuation domain with quotient field $K$. We denote by $v$ the valuation associated to $V$ and by $\Gamma_v$ the value group of $V$. We denote by $M$ the maximal ideal of $V$. The \emph{rank} of $V$ is the rank of its value group, which is equal to the Krull dimension of $V$.

Let $E:=\{s_n\}_{n\inN}$ be a sequence of elements of $K$. We say that $E$ is a \emph{pseudo-convergent sequence} if the sequence $\bdelta(E):=\{\delta_n:=v(s_{n+1}-s_n)\}_{n\inN}\subseteq\Gamma_v$ (called the \emph{gauge} of $E$) is strictly increasing. The \emph{breadth ideal} of $E$ is 
\begin{equation*}
\Br(E):=\{x\in K\mid v(x)>\delta_n\text{~for all~}n\inN\};
\end{equation*}
the breadth ideal is always a fractional ideal of $V$. An element $\alpha\in K$ is a \emph{pseudo-limit} of $E$ if $v(\alpha-s_n)=\delta_n$ for all $n\inN$; we denote the set of pseudo-limits of $E$ by $\inslim_E$. If $\inslim_E$ is nonempty, then $\inslim_E=\alpha+\Br(E)$ for any pseudo-limit $\alpha$ (\cite[Lemma 3]{kaplansky-maxfield}). We note that, in general, pseudo-convergent sequences can be indexed by any well-ordered set $\Lambda$ but that for our purposes it suffices to consider only those indexed by $\insN$ (see Remark \ref{oss:indexed}).



\section{The polynomial closure of a pseudo-convergent sequence}

The following lemma shows that, given a pseudo-convergent sequence $E=\{s_n\}_{n\inN}\subset K$, an element $t\in K$ can be close to at most one of the elements of $E$ (with respect to the gauge).

\begin{lemma}\label{not in open ball}
Let $E:=\{s_n\}_{n\inN}\subset K$ be a pseudo-convergent sequence with gauge $\{\delta_n\}_{n\inN}$, and let $t\in K$. Then, $v(s_n-t)\leq\delta_n$ for all but at most one $n\inN$.
\end{lemma}
\begin{proof}
Suppose $v(s_n-t)>\delta_n$, and let $s_m\in E$. If $m<n$, then
\begin{equation*}
v(s_m-t)=v(s_m-s_n+s_n-t)=\delta_m
\end{equation*}
since $v(s_m-s_n)=\delta_m<\delta_n<v(s_n-t)$; on the other hand, if $m>n$ then
\begin{equation*}
v(s_m-t)=v(s_m-s_n+s_n-t)=\delta_n<\delta_m
\end{equation*}
since $v(s_m-s_n)=\delta_n<v(s_n-t)$. The claim is proved.
\end{proof}

\begin{lemma}\label{greatest prime ideal in an ideal}
Let $I\subset M\subset V$ be an ideal. Then, the largest prime ideal contained in $I$ is equal to 
\begin{equation*}
\bigcap_{t\notin I,n\geq1}t^nV
\end{equation*}
\end{lemma}
\begin{proof}
Let $P(I):=\bigcap_{t\notin I,n\geq1}t^nV$. Then, $P(I)$ is a prime ideal by \cite[Theorem 17.1(3)]{gilmer}. If $\alpha\in P(I)\setminus I$, then $\alpha\in \alpha^n V$ for every $n$, which is not possible (unless $\alpha$ is a unit, which we can exclude since $I\subset M$). This shows that $P(I)\subseteq I$. 

Let $Q\subseteq I$ be a prime ideal. If for some $t\notin I$ there exists $n\in\N$ such that $t^n\in Q$ then $t\in Q\subseteq I$, a contradiction. Thus $Q\subseteq P(I)$, and $P(I)$ is the largest prime ideal contained in $I$.
\end{proof}
The previous lemma can also be rephrased by saying that $x\in P(I)$ if and only if $v(x)>nv(t)$ for all $t\in V\setminus I$ and all $n\in\N$.

\begin{prop}\label{prop:Pk}
Let $E:=\{s_n\}_{n\inN}\subset K$ be a pseudo-convergent sequence with gauge $\{\delta_n\}_{n\inN}$; let $c_n:=s_{n+1}-s_n$. Let $\alpha\in K$ and take any $k\inN$; let $P_k$ be the largest prime ideal contained in $c_k^{-1}\Br(E)$. Then the following are equivalent:
\begin{enumerate}[(i)]
    \item\label{prop:Pk:lambda} $v(\alpha-s_k)>\lambda(\delta_r-\delta_k)+\delta_k$ for every $r\geq k$ and every $\lambda\inN$;
    \item\label{prop:Pk:Br} $\alpha\in s_k+c_kP_k$.
\end{enumerate}
\end{prop}
\begin{proof}
Let $\beta:=\frac{\alpha-s_k}{c_k}$; then, $v(\beta)=v(\alpha-s_k)-\delta_k$, and thus we have to show that $\beta\in P_k$ if and only if $v(\beta)>\lambda(\delta_r-\delta_k)$ for every $\lambda\inN$ and $r\geq k$.

The sequence $F:=c_k^{-1}E=\{c_k^{-1}s_n\}_{n\inN}$ is pseudo-convergent with gauge $\{\delta_n-\delta_k\}_{n\inN}$, and thus $\Br(F)=c_k^{-1}\Br(E)\subsetneq V$. Hence, by Lemma \ref{greatest prime ideal in an ideal}, $\beta\in P_k$ if and only if $\beta\in t^{\lambda}V$ for every $t\in V\setminus\Br(F)$ and every $\lambda\in \N$. By definition, this is equivalent to $v(\beta)>\lambda(\delta_r-\delta_k)$ for every $r,\lambda\in\N$. Hence, the two conditions are equivalent.
\end{proof}

The following lemma is essentially \cite[Proposition 4.8]{chabert-polynomialclosure}; we prove it explicitly to show that it holds without any hypothesis on the rank of $V$.
\begin{lemma}\label{lemma:inslim}
Let $E:=\{s_n\}_{n\inN}$ be a pseudo-convergent sequence. Then, $\inslim_E\subseteq\overline{E}$.
\end{lemma}
\begin{proof}
Let $\alpha\in\inslim_E$, and let $f\in\Int(E,V)$; we can write it as $f(X)=\sum_ja_j(X-\alpha)^j$. By the proof of \cite[Proposition 3.7]{pseudoconv}, there is a $k$ such that, for all large $n$, $v(f(s_n))=v(a_k(s_n-\alpha)^k)<v(a_j(s_n-\alpha)^j)$ for all $j\neq k$. Since $v(f(s_n))\geq 0$ for all $n$, it follows that $v(f(\alpha))=v(a_0)\geq 0$. Hence $\alpha\in\overline{E}$.
\end{proof}

\begin{oss}\label{oss:indexed}
The previous lemma also shows why, in this context, it is enough to consider pseudo-convergent sequences indexed by $\insN$. Indeed, let $E:=\{s_\nu\}_{\nu\in\Lambda}$ be a pseudo-convergent sequences indexed by a well-ordered set $\Lambda$, and let $E_{\mathrm{in}}$ be the subsequence $\{s_n\}_{n\inN}$: then, $E_{\mathrm{in}}$ is again pseudo-convergent. Let $\nu\in\Lambda\setminus\insN$. Then, $s_\nu\in\inslim_{E_{\mathrm{in}}}\subseteq\overline{E_{\mathrm{in}}}$, and thus $\overline{E}=\overline{E_{\mathrm{in}}}$; hence, we do not lose anything by considering only $\overline{E_{\mathrm{in}}}$.
\end{oss}

For each $n\inN$, consider the polynomial
\begin{equation*}
H_n(X):=\prod_{i=0}^{n-1}\frac{X-s_i}{s_n-s_i}.
\end{equation*}
Note that for each $n$, $H_n(s_j)$ is zero for $j<n$ and is a unit of $V$ for $j\geq n$, as $v(s_j-s_i)=\delta_i=v(s_n-s_i)$ when $j\geq n> i$. In particular, these polynomials are integer-valued on $E$, and thus by \cite[Proposition 20]{survey} they form a \emph{regular basis} for $\Int(E,V)$, that is, a basis for the $V$-module $\Int(E,V)$ such that $\deg(H_n)=n$ for each $n\in\N$. In particular, an element $\alpha\in K$ is in $\overline{E}$ if and only if $H_n(\alpha)\in V$ for all $n\in\N$.
\begin{teor}\label{teor:chiuspol-pcv}
Let $E:=\{s_n\}_{n\inN}$ be a pseudo-convergent sequence with gauge $\{\delta_n\}_{n\inN}$; let $c_n:=s_{n+1}-s_n$. Then,
\begin{equation}\label{eq:ovE}
\overline{E}=\inslim_E\cup\bigcup_{n\geq 1}(s_n+c_nP_n),
\end{equation}
where $P_n$ is the largest prime ideal contained in $c_n^{-1}\Br(E)$. Furthermore, the union is disjoint.
\end{teor}
\begin{proof}
Suppose $\alpha\in\overline{E}$.

\medskip

If $v(\alpha-s_n)=\delta_n$ for every $n$ then $\alpha\in\inslim_E$, and in particular it is contained in the right hand side of \eqref{eq:ovE}. Suppose that is not the case: we distinguish two possibilities.

Suppose that $v(\alpha-s_n)\leq\delta_n$ for every $n\in\N$ and that $k$ is the smallest index for which $v(\alpha-s_k)<\delta_k$; in particular,  $v(\alpha-s_i)=\delta_i$ for all $i<k$. We have
$$v(H_{k+1}(\alpha))=\sum_{i=0}^k v(\alpha-s_i)-\sum_{i=0}^k \delta_i=v(\alpha-s_k)-\delta_k<0$$
a contradiction with the fact that $\alpha\in\overline{E}$.

Suppose now that $v(\alpha-s_k)>\delta_k$ for some $k$; by Lemma \ref{not in open ball} this $k$ is unique, and for all the other indexes we have
\begin{equation}\label{1}
v(\alpha-s_i)=v(\alpha-s_k+s_k-s_i)=\left\{
\begin{array}{cc}
\delta_i,&\text{ if }i<k\\
\delta_k,&\text{ if }i>k
\end{array}
\right.
\end{equation}
In particular, $v(H_{k+1}(\alpha))=v(\alpha-s_k)-\delta_k>0$ and if $n>k+1$ by \eqref{1} we have
\begin{align}\label{2}
v(H_n(\alpha))&=\sum_{i=0}^{k-1}(\delta_i-\delta_i)+v(\alpha-s_k)-\delta_k+\sum_{i=k+1}^{n-1}(\delta_k-\delta_i)=\nonumber\\
&=v(\alpha-s_k)-\delta_k+\sum_{i=k+1}^{n-1}(\delta_k-\delta_i)
\end{align}
Let now $\lambda,m\in\N$, $m\geq k$, be fixed. Choose $n$ so that $n-m>\lambda$. In particular,  $\sum_{i=k+1}^{n-1}\delta_i>\lambda\delta_m$. Hence, by \eqref{2} and the fact that $H_n(\alpha)\in V$ we have
$$v(\alpha-s_k)-\delta_k\geq \sum_{i=k+1}^{n-1}(\delta_i-\delta_k)>\lambda(\delta_m-\delta_k)$$
Since $\lambda,m$ are arbitrary, by Proposition \ref{prop:Pk} it follows that $\alpha\in s_k+c_kP_k$, as we wanted to show.

\medskip

Let now $\alpha$ be in the right hand side of \eqref{eq:ovE}. If $\alpha\in\inslim_E$ then $\alpha\in\overline{E}$ by Lemma \ref{lemma:inslim}. Suppose that $\alpha\notin\inslim_E$ and $\alpha\in s_k+c_kP_k$ for some $k\geq 1$: then by Proposition \ref{prop:Pk}  $v(\alpha-s_k)>\lambda(\delta_r-\delta_k)+\delta_k$ for every $r\geq k$ and every $\lambda\inN$.

In order to show that $\alpha\in\overline{E}$, it is enough to prove that $H_n(\alpha)\in V$ for all $n\in\N$.

If $n\leq k$, then by \eqref{1} we have
$$v(H_n(\alpha))=\sum_{i=0}^{n-1}(\delta_i-\delta_i)=0.$$
For $n=k+1$ we have $v(H_{k+1}(\alpha))>0$ as we remarked above.

Suppose now that $n>k+1$. Then by \eqref{2} we have
\begin{align*}
v(H_n(\alpha))&>v(\alpha-s_k)-\delta_k+\sum_{i=k+1}^{n-1}(\delta_k-\delta_{n-1})=\\
&=v(\alpha-s_k)-\delta_k+(n-k-1)(\delta_k-\delta_{n-1})
\end{align*}
and the last quantity is greater than zero by assumption. Hence, $\alpha\in\overline{E}$.

\medskip

We conclude the proof of the theorem by showing the last claim. For every $t\in s_k+c_kP_k$, we have $t-s_k\in c_kP_k$ and in particular  $v(t-s_k)>v(c_k)=\delta_k$. In particular, no such $t$ can be a pseudo-limit of $E$ (since otherwise $v(t-s_k)=\delta_k$ for every $k$). Moreover, if $t\in(s_k+c_kP_k)\cap(s_{k'}+c_{k'}P_{k'})$ for some $k'>k$, then we should have at the same time $v(t-s_k)>\delta_k$ and $v(t-s_{k'})>\delta_{k'}$, in contradiction with Lemma \ref{not in open ball}. Hence the union is disjoint, as claimed.
\end{proof}

As a consequence of Theorem \ref{teor:chiuspol-pcv}, we have the main result of the paper.
\begin{teor}
Let $V$ be a valuation domain of rank $>1$. Then, the polynomial closure is not a topological closure.
\end{teor}
\begin{proof}
Since $V$ has rank bigger than $1$, there is a nonmaximal prime ideal $P'$; if $t\in V\setminus P'$ is a nonunit, then the largest prime ideal $P$ (strictly) contained in $tV$ is different from the zero ideal. Let $E:=\{t^n\}_{n\inN}$ and let $E':=\{t^n\}_{n\geq 2}$. Then, $E$ and $E'$ are pseudo-convergent sequences with breadth ideal $\Br(E)=\Br(E')=P$ and with $\inslim_E=\inslim_{E'}$. Moreover, for every $n$, we have $(t^{n+1}-t^n)^{-1}P=P$.

By Theorem \ref{teor:chiuspol-pcv}, it follows that
\begin{equation*}
\overline{E}=\inslim_E\cup\bigcup_{k\geq 1}(t^k+P)=\overline{E'}\cup(t+P);
\end{equation*}
moreover, the first union is disjoint, and so $(t+P)\cap\overline{E'}=\emptyset$. If the polynomial closure were topological, we would have $\overline{E}=\overline{E'\cup\{t\}}=\overline{E'}\cup\overline{\{t\}}=\overline{E}\cup\{t\}$ (since finite sets are polynomially closed \cite[Chapter IV, Example IV.1.3]{CaCh}); however, for every $p\in P\setminus\{0\}$, the element $t+p$ is in $\overline{E}$ but not in $\overline{E'}\cup\overline{\{t\}}$. Thus, the polynomial closure is not topological, as claimed.
\end{proof}

To conclude the paper, we show when two pseudo-convergent sequences have the same polynomial closure.
\begin{prop}\label{prop:IntEIntF}
Let $E:=\{s_n\}_{n\inN}$ and $F:=\{t_n\}_{n\inN}$ be two pseudo-convergent sequences with gauges $\bdelta(E):=\{\delta_n\}_{n\inN}$ and $\bdelta(F):=\{\eta_n\}_{n\inN}$, respectively. Then, $\overline{E}=\overline{F}$ if and only if $\delta_t=\eta_t$ for every $t$ and $v(t_k-s_k)>\lambda(\delta_r-\delta_k)+\delta_k$ for every $k\geq r$ and every $\lambda\inZ$.
\end{prop}
\begin{proof}
By Theorem \ref{teor:chiuspol-pcv}, we can write
\begin{equation*}
\overline{E}=\inslim_E\cup\bigcup_{k\inN}(s_k+c_kP_k)
\end{equation*}
and
\begin{equation*}
\overline{F}=\inslim_F\cup\bigcup_{k\inN}(t_k+d_kQ_k)
\end{equation*}
for some $c_k,d_k\in K$ and prime ideals $P_k,Q_k$ defined as in the theorem. Let $S_k:=s_k+c_kP_k$ and $T_k:=t_k+d_kQ_k$.

Suppose that the two conditions of the statement hold. Then, by Proposition \ref{prop:Pk}, for every $k$ we have $v(c_k)=v(d_k)$, $P_k=Q_k$ and $s_k-t_k\in c_kP_k$, so that $S_k=T_k$. Furthermore, $s_k-t_k\in\Br(E)$ for every $k$, and thus $E$ and $F$ are equivalent in the sense of \cite[Section 5]{pseudoconv}, so $\inslim_E=\inslim_F$ by \cite[Lemma 5.3]{pseudoconv} and $\overline{E}=\overline{F}$.

Conversely, suppose $\overline{E}=\overline{F}$. Let $x,y\in\overline{E}$: then
\begin{itemize}
\item if $x,y\in S_k$ then $x-y\in c_kP_k\subseteq\Br(E)$;
\item if $x\in S_k$ and $y\in S_j$ for $k<j$ then $v(x-y)=\delta_k$;
\item if $x\in S_k$ and $y\in\inslim_E$ then $v(x-y)=\delta_k$;
\item if $x,y\in\inslim_E$ then $x-y\in\Br(E)$.
\end{itemize}
Let $D(E):=\{v(x-y)\mid x,y\in\overline{E}\}$: then, $D(E)=\bdelta(E)\cup X_E$, where $X_E$ is an up-closed subset of $\Gamma_v\cup\{\infty\}$ (more precisely, $X=v(\Br(E))$ if $\inslim_E$ has at least two elements, while $X=\bigcup_iv(c_iP_i)$ otherwise). Analogously, $D(F)=\bdelta(F)\cup X_F$.

If $\overline{E}=\overline{F}$, then $D(E)=D(F)$. Since $X_E$ is the largest up-closed subset of $D(E)$ (and analogously for $D(F)$), we must have $\bdelta(E)=\bdelta(F)$; since the gauges are linearly ordered, it must be $\delta_n=\eta_n$ for every $n\inN$. In particular, $\Br(E)=\Br(F)$ and $v(c_k)=v(d_k)$ for every $k$; thus, $P_k=Q_k$ for every $k$.

Therefore, to prove the statement we only need to show that $s_k\in T_k$ for every $k$. For $y\in\overline{E}$, let $D(E,y):=\{v(x-y)\mid x\in\overline{E}\}$: then, with the same reasoning as above, we see that
\begin{equation*}
D(E,y)=\begin{cases}
\{\delta_1,\ldots,\delta_n\}\cup X_{E,y} & \text{if~}y\in S_k\\
D(E) & \text{if~}y\in\inslim_E,
\end{cases}
\end{equation*}
where $X_{E,y}$ is an up-closed subset of $\Gamma_v\cup\{\infty\}$. Clearly, $D(E,y)=D(F,y)$; in particular, $D(F,s_k)=\{\delta_1,\ldots,\delta_k\}\cup X_{E,s_k}$, and thus it must be $s_k\in T_k$. The claim is proved.
\end{proof}

When $V$ has rank 1, the two previous propositions have a very simplified form, which can also be obtained from Chabert's paper \cite{chabert-polynomialclosure}.
\begin{cor}\label{cor:chiuspol-rk1}
Let $V$ be a valuation ring of rank $1$, and let $E:=\{s_n\}_{n\inN}$ and $F:=\{t_n\}_{n\inN}$ be two pseudo-convergent sequences. Then:
\begin{enumerate}[(a)]
    \item\label{cor:chiuspol-rk1:chiuspol} $\overline{E}=E\cup\inslim_E$;
    \item\label{cor:chiuspol-rk1:EF} $\Int(E,V)=\Int(F,V)$ if and only if $s_n=t_n$ for every $n\inN$.
\end{enumerate}
\end{cor}
\begin{proof}
Each $c_k^{-1}\Br(E)$ is a proper, non-maximal ideal of $V$; therefore, if $V$ has rank $1$ then we must have $P_k=(0)$. Hence, \ref{cor:chiuspol-rk1:chiuspol} follows from Theorem \ref{teor:chiuspol-pcv}, while \ref{cor:chiuspol-rk1:EF} from Proposition \ref{prop:IntEIntF}, since we must have $s_k-t_k\in P_k=(0)$.
\end{proof}

\subsection*{Acknowledgments}

The authors wish to warmly thank the referee for suggesting a different strategy for the proof of the main Theorem \ref{teor:chiuspol-pcv}, which greatly improved its clarity.

\end{document}